\documentclass{article}

\usepackage{amsthm}
\usepackage{amsmath,amscd}
\usepackage{amsfonts}
\usepackage{amssymb}
\usepackage[T1]{fontenc}

\newcommand{\cY}{\mathcal{Y}}
\newcommand{\gP}{\textbf{P}}
\newcommand{\bN}{\mathbb{N}}
\newcommand{\bP}{\mathbb{P}}
\newcommand{\bC}{\mathbb{C}}
\newcommand{\cO}{\mathcal{O}}
\newcommand{\cR}{\mathcal{R}}
\newcommand{\Om}{\Omega}
\newcommand{\cX}{\mathcal{X}}
\newtheorem*{thm}{Theorem}
\newtheorem*{thmglobal}{Theorem \ref{global}}

\newtheorem*{lem}{Lemma}
\newtheorem{theorem}{Theorem}[section]
\newtheorem{lemma}[theorem]{Lemma}
\newtheorem{e-proposition}[theorem]{Proposition}
\newtheorem{corollary}[theorem]{Corollary}
\newtheorem{e-definition}[theorem]{Definition\rm}

\newtheorem{remark}{\it Remark\/}

\DeclareMathOperator{\Bs}{Bs}
\DeclareMathOperator{\D}{D}
\title{Hyperbolicity Related Problems for Complete Intersection Varieties}
\author{Damian Brotbek}

\begin{document}
\maketitle
\begin{abstract}
In this paper we examine different problems regarding complete intersection varieties of high degree in a complex projective space. First we show how one can deduce hyperbolicity for generic complete intersection of high multidegree and high codimension from the known results on hypersurfaces. Then we prove a existence theorem for jet differentials that generalizes a theorem of S. Diverio. Finally, motivated by a conjecture of O. Debarre, we focus on the positivity of the cotangent bundle of complete intersections, and prove some results towards this conjecture; among other things, we prove that a generic complete intersection surface of high multidegree in a projective space of dimension at least four has ample cotangent bundle. 
\end{abstract}

%
%

\section{Introduction}

Varieties with ample cotangent bundle satisfy many interesting properties and are supposed to be abundant. However few examples of such varieties are known. 
In \cite{Deb05}, O. Debarre constructed such varieties by proving that the intersection of at least $\frac{N}{2}$ sufficiently ample general hypersurfaces in an $N$-dimensional abelian variety has ample cotangent bundle.
And motivated by this result, he conjectured that the analogous statement holds in projective space, that is to say, \emph{the intersection, in $\bP^N$, of at least $\frac{N}{2}$ generic hypersurfaces of sufficiently high degree has ample cotangent bundle}. 

It is a well known fact that varieties with ample cotangent bundle are hyperbolic in the sense of Kobayashi. In this sense, this conjecture on complete intersection varieties relates to Kobayashi's conjecture for hypersurfaces which predicts that a generic hypersurface of sufficiently high degree in $\bP^N$ is hyperbolic. 
A strategy towards this conjecture was explained by Y-T. Siu in \cite{Siu04}, and was afterwards further developed by S. Diverio, J. Merker and E. Rousseau in \cite{DMR10} to get effective algebraic degeneracy for entire curves on generic hypersurfaces of high degree.
Moreover S. Diverio and S. Trapani (\cite{D-T09}) generalized Debarre's conjecture to complete intersection varieties of any codimension by replacing the cotangent bundle by a suitable invariant jet differential bundle to give an unified point of view on those questions.\\

In this paper we study some aspects of those problems, and give some partial results towards Debarre's conjecture. Our main results can be summarized in the following statement.

\begin{theorem}
Let $X\subseteq \bP^N$ be a smooth complete intersection variety of dimension $n$, codimension $c$ and multidegree $(d_1,\cdots, d_c)$. Let $A$ be an ample divisor on $X$ and suppose that the $d_i$'s are big enough. Then

\begin{enumerate}
\item For $k\geq \lceil \frac{n}{c} \rceil$ and  $m \gg 0$ there is a non zero invariant jet differential of order $k$ and weight $m$ (ie, $H^0(X,E_{k,m}\Om_X\otimes A^{-1})\neq 0$). \label{MSjet}
\item If $2c\geq n$ and $X$ is generic then $X$ is hyperbolic. \label{MShyperbolic}
\item If $c\geq n$ then $\Om_X$ is numerically positive.\label{MSnumpos}
\item If $c\geq n$ and $X$ is generic then $\Om_X$ is big and almost everywhere ample.\label{MSAEample}
\item If $n=2$, $N\geq 4$ and $X$ is generic then $\Om_X$ is ample.\label{MSample}
\end{enumerate}
\end{theorem}

The paper is divided into three parts. The first part (section \ref{hyperbolicity}) is focused on the proof of \ref{MShyperbolic}, this is a consequence of the results on hypersurfaces proved in \cite{DMR10} and \cite{D-T09}. In the second part (section \ref{sectionnonannulation}) we prove statement \ref{MSjet}, which is a generalization of a non vanishing proved by S. Diverio \cite{Div09}. The last part (section \ref{positivitecotangent}) is dedicated on giving partial results toward Debarre's conjecture. We prove statement \ref{MSnumpos} in section \ref{numericalpositivity}, then in section \ref{mainresult} we will adapt the technics developed in \cite{DMR10} to prove statements \ref{MSAEample} and \ref{MSample}.

%
%

\section{Algebraic Degeneracy and Hyperbolicity for Complete intersections}\label{hyperbolicity}

Here we show that from the work of \cite{DMR10} and \cite{D-T09} one can straightforwardly deduce the hyperbolicity of generic complete intersections of high codimension and of high multidegree, just by "moving" the hypersurfaces we are intersecting. We begin with a definition.


\begin{e-definition}
Let $X$ be a projective variety, we define the algebraic degeneracy locus to be the Zariski closure of the union of all non-constant entire curves $f:\mathbb{C}\to X$ 
$$dl(X):= \overline{\bigcup{f(\mathbb{C})}}.$$
\end{e-definition}

Recall the main result proven in \cite{DMR10} and \cite{D-T09}.

\begin{thm}\label{D-T}
There exists $\delta \in \mathbb{N}$ such that if $H$ is a generic hypersurface of degree $d\geq \delta$, then there exists a proper algebraic subset $Y \subset H$ of codimension at least two in $H$ such that  $dl(H)\subset Y$.
\end{thm}


\begin{remark}\label{zariski}
In fact they prove something slightly stronger. Consider the universal hypersurface of degree $d$ in $\mathbb{P}^N$
$$\mathcal{H}_d=\left\{(x,t)\in \mathbb{P}^N \times \mathbb{P}^{N_d}\;\;\; /\;\;\; x\in H_{d,t} \right\}$$
where $\mathbb{P}^{N_d}:=\mathbb{P}(H^0(\mathbb{P}^N,\mathcal{O}_{\mathbb{P}^N}(d))^*)$ and $H_{d,t}=(t=0)$. Denote $\pi_d$ the projection on the second factor. 
During the proof of Theorem $\ref{D-T}$ it is shown that for $d\gg 0$ there exists an open subset $U_d\subset \mathbb{P}^{N_d}$ and an algebraic subset $\mathcal{Y}_d\subset \mathcal{H}_{d|U_d}\subset U_d\times\mathbb{P}^N$ such that for all $t\in U_d$, the fibre $Y_{d,t}$ has codimension 2 in $H_{d,t}$ and $dl(H_{d,t})\subset Y_{d,t}$.
\end{remark}


\begin{remark}
A major result of \cite{DMR10} is that $\delta$ is effective. They prove that $\delta \leqslant 2^{(N-1)^5}$. In what follows, any bound found on $\delta$ could be used.
\end{remark}


We are going to use the standard action of $G:= Gl_{N+1}(\mathbb{C})$ on $\mathbb{P}^N$. For any $g\in G$ and any variety $X\subseteq \mathbb{P}^N$ we write $g\cdot X:=g^{-1}(X)$.


\begin{remark}\label{rkaction} 
Let $g\in G$ and $X \subset \mathbb{P}^N$ a projective variety. If $f:\mathbb{C}\to g\cdot X$ is a non-constant entire curve then $g\circ f :\mathbb{C}\to X$ is a non-constant entire curve, therefore $g\circ f (\mathbb{C})\subseteq Y$ and thus $f(\mathbb{C})\subseteq g\cdot Y$. This proves that $g\cdot dl(X)= dl(g\cdot X)$
\end{remark}


\begin{remark}\label{cap}
Note also that if $X_1$ and $X_2$ are two projective varieties in $\mathbb{P}^N$, then $dl(X_1\cap X_2)\subseteq dl(X_1)\cap dl(X_2)$. 
\end{remark}


We will combine these remarks with the following moving lemma.


\begin{lemma}\label{moving}
Let $V\subset \mathbb{P}^N$ and $W\subset \mathbb{P}^N$ be algebraic subsets such that $\dim (V)=~n$ and $\dim (W)= m$. Then for a generic $g\in G$, we have  $$\dim (g\cdot V \cap W)= \max\left\{n+m-N,0\right\}.$$
\end{lemma}


\begin{proof}
We will proceed by induction on $m$. The case $m=0$ is clear.
Now for the general case, let $H$ be a hyperplane of $\mathbb{P}^N$, such that $\dim(H\cap W)\leqslant m-1$ (such a hyperplan clearly exists). Now we take a generic $g\in G$, such that, by induction hypothesis, $\dim(g\cdot V \cap (W \cap H))= \max\left\{n+m-1-N,0\right\}$. Now argue by contradiction, obviously $\dim(g\cdot V \cap W)\geq \max\left\{n+m-N,0\right\}$. If $\dim(g\cdot V \cap W)> \max\left\{n+m-N,0\right\}$ then 
there exists an irreducible $Z\subset g\cdot V \cap W$ such that $\dim (Z)=\alpha>\max\left\{n+m-N,0\right\}$. But then $\dim (Z\cap H)\geqslant \alpha -1$, and since $Z\cap H \subset g\cdot V \cap (W\cap H)$, we get $\max\left\{n+m-N-1,0\right\} < \alpha-1 \leqslant \dim(Z\cap H) \leqslant \dim (g\cdot V \cap (W\cap H)) \leqslant \max\left\{n+m-1-N,0\right\}$, which yields the desired contradiction.
\end{proof}

\begin{corollary}\label{pseudogeneric}
Let $c\geqslant 1$, take $c$ hypersurfaces $H_1,\cdots ,H_c \subset \mathbb{P}^N$ satisfying the conclusions of Theorem $\ref{D-T}$ and take generic $g_2\cdots ,g_c\in G$. Consider $X:= H_1 \cap g_2\cdot H_2 \cap \cdots \cap g_c\cdot H_c$ (which is a smooth complete intersection). Then there is an algebraic subset $Y \subset X$ of dimension $N-3c$ in $X$, such that $dl(X)\subseteq Y $. In particular $X$ is hyperbolic if $c\geqslant \frac{N}{3}$. 
\end{corollary}
\begin{proof}
This is a simple induction on $c$. The case $c=1$ is just the content of Theorem $\ref{D-T}$. Let $c>0$ and suppose that $X_{c-1}:= H_1 \cap g_2\cdot H_2 \cap \cdots \cap g_c\cdot H_{c-1}$ is a smooth complete intersection that contains a proper subset $Y_{c-1}$ of dimension $N-3c+3$ in $X$, such that $Y_{c-1}$ contains the image of all the entire curves in $X$. By hypothesis there is an algebraic subset $Y_c\subset X_c$ of dimension $N-3$, such that $dl(H_c)\subseteq Y_c$.  Now for generic $g_c\in G$, $X_{c-1}\cap g_c\cdot H_c$ is a smooth complete intersection and moreover on can apply Lemma $\ref{moving}$ to $Y_{c-1}$ and $Y_c$ and thus
 
$$\dim (g_c\cdot Y_c  \cap Y_{c-1})= N-3+N-2c+1-N= N-3c.$$
Now, Remarks \ref{rkaction} and \ref{cap} yield $dl(X)\subseteq Y$.
\end{proof}

\begin{corollary}\label{HCI}
Let $X=H_1\cap \cdots \cap H_c \subset \mathbb{P}^N$, $H_i \in H^0(\mathbb{P}^N,\mathcal{O}_{\mathbb{P}^N}(d_i))$ be a generic smooth complete intersection such that $d_i\geq \delta $. Then there is an algebraic subset $Y \subset X$ of dimension $N-3c$ in $X$, such that $dl(X)\subseteq Y $. In particular $X$ is hyperbolic if $c\geqslant \frac{N}{3}$. 
\end{corollary}
\begin{proof}
With the notations of remark $\ref{zariski}$, consider $U=U_{d_1}\times \cdots \times U_{d_c}$, $\mathcal{Y}= \mathcal{Y}_{d_1}\cap \cdots \cap \mathcal{Y}_{d_1} $. It is well known that $\dim(Y_t)$ is an upper semi continuous function on $U$. Therefore we only need to show that there is no Zariski open subset in $U$ on which $\dim(Y_t)>N-3c$, but this is clear by corollary $\ref{pseudogeneric}$.
\end{proof}

%
%

\section{Jet differentials on complete intersection varieties}\label{sectionnonannulation}

This part is motivated by two theorems of S. Diverio (see \cite{Div08} and \cite{Div09}).
In \cite{Div09}, he proved a nonvaninshing theorem for jet differentials on hypersurfaces of high degree.

\begin{thm}\label{nonvanishing}
\emph{(\cite{Div09} Theorem 1)}
Let $X\subset \mathbb{P}^{n+1}$ a smooth projective hypersurface and $A$ an ample line bundle on $X$. Then there exists a positive integer $d$ such that 
$$H^0(X, E_{k,m}\Omega_X\otimes A^{-1})\neq 0$$
if $k\geqslant n$, $\deg(X)\geqslant d$ and $m$ large enough.
\end{thm} 

He also proved a vanishing theorem for jet differentials on complete intersection varieties.

\begin{thm}\label{vanishing} 
\emph{(\cite{Div08} Theorem 7)}
Let $X$ be a complete intersection variety in $\mathbb{P}^N$ of dimension $n$ and codimension $c$.
For all $m\geqslant 1$ and $1\leqslant k < \lceil\frac{n}{c}\rceil$ one has
$$H^0(X,E_{k,m}\Omega_X)=0.$$ 
\end{thm}

It seems therefore natural to look for the nonvanishing of $H^0(X,E_{\lceil\frac{n}{c}\rceil,m}\Omega_X)$ when $X$ is a smooth complete intersection of dimension $n$ and codimension $c$ of high multidegree. This is the content of Theorem $\ref{thmjet}$.

\subsection{Segre Classes and Higher order jet spaces}
We start by giving the definition of the Segre classes associated to a vector bundle.
If $E$ is a rank $r$ complex vector bundle on $X$ and $p:\mathbb{P}(E)\to X$ the projection, then the Segre classes of $E$ are defined by $$s_i(E):=p_*c_1(\mathcal{O}_{\mathbb{P}(E)}(1))^{r-1+i}.$$ 
(Note that it is denoted $s_i(E^*)$ in $\cite{Ful98}$). It is straightforward to check that for any line bundle $L\to X$
\begin{equation}\label{segredroite}
s_i(E\otimes L)= \sum_{j=0}^{i} \binom{r-1+i}{i-j}s_j(E)c_1(L)^{i-j}.
\end{equation} 
 
Recall that the total Segre class is the formal inverse of the total Chern class of the dual bundle, $s(E)=c(E^*)^{-1}$. Therefore total Segre classes satisfy Whitney's formula for vector bundle exact sequences.

Now we briefly recall the construction of higher order jet spaces , details can be found in \cite{Dem97} and \cite{Mou10}, we will follow the presentation of $\cite{Mou10}$. Let $X\subset \mathbb{P}^N$ be a projective variety of dimension $n$, for all $k\in \mathbb{N} $ we can construct a variety $X_k$ and a rank $n$ vector bundle $\mathcal{F}_k$ on $X_k$. Inductively,  $X_0:=X$ and $\mathcal{F}_0:=\Omega_X$. Let $k\geqslant 0$ suppose that $X_k$ and $\mathcal{F}_k$ are constructed, then  $X_{k+1}:=\mathbb{P}(\mathcal{F}_{k})\overset{\pi_{k,k+1}}{\longrightarrow} X_k$ and $\mathcal{F}_{k+1}$ is the quotient of $\Omega_{X_{k+1}}$ defined by the following diagram.\\

$\begin{CD}
  @.            0                @.                 0          @.                                 @.          \\
@.            @VVV                                @VVV                          @.                          @.\\
  @.         S_{k+1}             @=              S_{k+1}       @.                                 @.          \\   
@.            @VVV                                @VVV                          @.                          @.\\ 
0 @>>>\pi^*_{k,k+1}\Omega_{X_k} @>>>     \Omega_{X_{k+1}}       @>>>       \Omega_{X_{k+1}/X_{k}}   @>>>    0\\  
@.            @VVV                                @VVV                          @|                          @.\\ 
0 @>>>\mathcal{O}_{X_{k+1}}(1)  @>>>    \mathcal{F}_{k+1}      @>>>       \Omega_{X_{k+1}/X_{k}}   @>>>     0 \\
@.            @VVV                                @VVV                          @.                          @.\\
  @.           0                @.                 0            @.                                 @.         \\ \\
\end{CD}$

For all $k > j \geqslant 0$ we will denote $\pi_{j,k}=\pi_{j,j+1}\circ \cdots \circ \pi_{k-1,k}: X_k\to X_j$, $\pi_k:=\pi_{0,k}$ and $E_{k,m}\Omega_X={\pi_k}_*\mathcal{O}_{X_k}(m)$. The bundles $E_{k,m}\Omega_X$ have important applications to hyperbolicity problems. 

Note that $n_k:=\dim(X_k)=n+k(n-1)$. 
If we have a k-uple of integers $(a_1,\cdots ,a_k)$ we will write 
$$\mathcal{O}_{X_k}(a_1,\cdots, a_k)=\pi_{1,k}^{*}\mathcal{O}_{X_1}(a_1)\otimes \cdots \otimes \mathcal{O}_{X_k}(a_k)$$


We define $s_{k,i}:=s_i({\mathcal{F}_k})$, $s_i:=s_i(\Omega_X)$, $u_k:=c_1(\mathcal{O}_{X_k}(1))$, $h_{\mathbb{P}^N}:=c_1(\mathcal{O}_{\mathbb{P}^N}(1))$, $h:={h_{\mathbb{P}^N}}_{|X}$ and $\mathcal{C}_k(X)=\mathbb{Z}\cdot u_k \oplus \cdots \oplus \mathbb{Z}\cdot u_1 \oplus \mathbb{Z}\cdot h \subset NS^1(X_k)$.
To ease our computations we will also adopt the following abuses of notations, if $k>j$ we will write $u_j$ the class on $X_k$ instead of $\pi^*_{j,k}u_j$ and similarly $s_{j,i}$ instead of $\pi^*_{j,k}s_{j,i}$. This should not lead to any confusion.

Now, from the horizontals exact sequences in the diagram, Whitney's formula and formula $\ref{segredroite}$ one easily derives (as in $\cite{Mou10}$)
the recursion formula
\begin{equation}\label{formulasegre}
s_{k,\ell}=\sum_{j=0}^{\ell}M_{\ell,j}^ns_{k-1,j}u_k^{\ell-j},
\end{equation}
where $M_{\ell,j}^n=\sum_{i=0}^{\ell-j}(-1)^i\binom{n-2+i+j}{i}$, in particular $M_{\ell,\ell}^n=1$.

\begin{lemma}\label{labase}
Let $k\geqslant 0$, $a\geqslant 0$, $\ell\geqslant 0$, take $\ell$ positive integers $i_1,\cdots, i_\ell$ and $m$ divisor classes $\gamma_1,\cdots, \gamma_m \in \mathcal{C}_k(X)$ such that $i_1+\cdots +i_\ell+m+a=n_k$. Denote $\gamma_q:=\alpha_{q,0}h+\sum_{i} \alpha_{q,i} u_i$. Then

$$ \int_{X_k}s_{k,i_1}\cdots s_{k,i_\ell} \gamma_1\cdots \gamma_m h^a =
\sum_{j_1,\cdots,j_{k+\ell},b} Q_{j_1,\cdots,j_{k+\ell},b}\int_X s_{j_1}\cdots s_{j_{k+\ell}} h^{a+b}
$$
Where in each term of the sum we have, $b\geqslant 0$ and the $Q_{j_1,\cdots,j_{k+\ell},b}$'s are polynomials in the $\alpha_{q,i}$'s whose coefficients are independant of $X$. Moreover up to reordering of the $j_p$'s one has $j_1\leqslant i_1,\cdots j_\ell\leqslant i_\ell$. 
\end{lemma}

\begin{proof}
This is an immediate induction on $k$. The result is clear for $k=0$. Now suppose it is true for some $k>0$ and take $m$ divisors, $\gamma_1,\cdots, \gamma_m \in \mathcal{C}_{k+1}(X)$  on $X_{k+1}$. Let $\gamma_q:=\alpha_{q,0}h+\sum_{i} \alpha_{q,i} u_i$. Then using recursion formula \ref{formulasegre} and expanding, we get
\begin{eqnarray}
\int_{X_{k+1}}\!\!\!\!\!\!\!s_{k+1,i_1}\cdots s_{k+1,i_\ell}\gamma_1\cdots \gamma_m h^a\!\!\!\!\!&=&\!\!\!\!\! \sum P_{j_i,\cdots, j_\ell,b}^{p_1,\cdots, p_{k+1}}\!\!\!\int_{X_{k+1}}\!\!\!\!\!\!\!s_{k,j_1}\cdots s_{k,j_\ell}u_{k+1}^{p_{k+1}}\cdots u_1^{p_1}h^{a+b} \nonumber \\
&=&\!\!\!\!\! \sum P_{j_i,\cdots, j_\ell,b}^{p_1,\cdots, p_{k+1}}\!\!\!\int_{X_{k}}\!\!\!\!\!s_{k,j_1}\cdots s_{k,j_\ell} s_{k,r} u_{k}^{p_{k}}\cdots u_1^{p_1}h^{a+b} \nonumber
\end{eqnarray}
where in each term of the sum, $r=p_{k+1}-(n-1)$ and moreover, thanks to formula $\ref{formulasegre}$ one has $j_1\leqslant i_1,\cdots, j_\ell\leqslant i_\ell$. Note also that the $P_{j_i,\cdots, j_\ell,b}^{p_1,\cdots, p_{k+1}}$'s are polynomials in the $\alpha_{i,j}$'s but their coefficients do not depend on $X$. Now we can conclude by induction hypothesis.
\end{proof}


\subsection{Segre classes for complete intersections}\label{sectionsegre}

From now on, we will take $X=H_1\cap\cdots \cap H_c\subset \mathbb{P}^N$ a complete intersection of dimension $n$ and $H_i=(\sigma_i=0)$ with $\sigma_i \in H^0(\mathbb{P}^N,\mathcal{O}_{\mathbb{P}^N}(d_i))$ and $d_i \in \mathbb{N}$. Note that $n+c=N$. Let $\kappa:= \lceil\frac{n}{c}\rceil$ and take $b$ such that $n=(\kappa -1)c+b$, observe that $0<b\leqslant c$. To simplify some of our formulas, we will also define $\hat{\imath}:=i+n-1$ so that ${\pi_{k-1,k}}_*u_k^{\hat{\imath}}=s_{k-1,i}$.
Moreover as we will be interested in the asymptotic behavior of polynomials in $\mathbb{Z}[d_1,\cdots , d_c]$ we need some more notations. Let $P\in\mathbb{Z}[d_1,\cdots , d_c]$, $\deg{P}$ denotes the degree of the polynomial and $P^{dom}$ the homogenous part of $P$ of degreee $\deg P$ . We will write $P=o(d^k)$ if $\deg{P}<k$ and if $Q\in \mathbb{Z}[d_1,\cdots , d_c]$ is another polynomial we will write $P\sim Q$ if $P^{dom}=Q^{dom}$ and $P\gtrsim Q$ if $P^{dom}\geqslant Q^{dom}$.  Moreover some of our computations will take place in the chow ring $A^*(X)$ so we introduce some more notations. If $P$ is a polynomial in $\mathbb{Z}[d_1,\cdots , d_c, h]$ homogenous of degree $k$ in $h$ we will write $\tilde{P}$ for the unique polynomial in $\mathbb{Z}[d_1,\cdots , d_c]$ satisfying $P=\tilde{P}h^k$.
 
First we compute the Segre classes of the twisted cotangent bundle of a complete intersection in $\mathbb{P}^N$. Let $m\in \mathbb{Z}$. The twisted Euler exact sequence
\begin{eqnarray}
0 \to \mathcal{O}_{\mathbb{P}^N}(-m) \to \mathcal{O}_{\mathbb{P}^N}^{\oplus N+1}(1-m) \to T\mathbb{P}^N(-m) \to 0 \nonumber
\end{eqnarray}
yields:
\begin{eqnarray}
c(T\mathbb{P}^N(-m))=\frac{c(\mathcal{O}_{\mathbb{P}^N}^{\oplus N+1}(1-m))}{c(\mathcal{O}_{\mathbb{P}^N}(-m))}
=\frac{(1+(1-m)h_{\mathbb{P}^N})^{N+1}}{(1-mh_{\mathbb{P}^N})} \nonumber
\end{eqnarray}
The (twisted) normal bundle exact sequence 

\begin{eqnarray}
0 \to TX(-m) \to T\mathbb{P}^N_{|X}(-m) \to \bigoplus_{i=1}^c \mathcal{O}_{X}(d_i-m)\to 0 \nonumber
\end{eqnarray}
yields 
\begin{eqnarray}
c(TX(-m))=\frac{c(T\mathbb{P}^N_{|X}(-m))}{c(\bigoplus_{i=1}^c \mathcal{O}_{X}(d_i-m))}
=\frac{(1+(1-m)h)^{N+1}}{(1-mh)\prod_{i=1}^c(1+(d_i-m)h)} \nonumber
\end{eqnarray}
Therefore:
\begin{eqnarray}
s(\Omega_X(m))=(1\!-\!(1\!-\!m)h\!+\!(1\!-\!m)^2h^2\!-...)^{N+1}(1\!-\!mh)\prod_{i=1}^c(1\!+\!(d_i\!-\!m)h) \nonumber
\end{eqnarray}

Expanding the right hand side as a polynomial in $\mathbb{Z}[d_1,\cdots,d_c,h]$ we see that for $\ell \geq c$ we have $\deg(\tilde{s_\ell})=c$ and that for ${\ell}\leqslant c$ we have
\begin{eqnarray}\label{dominant}
s_{\ell}^{dom}(\Omega_X(m)) =\sum_{j_1<...<j_{\ell}} d_{j_1}...d_{j_{\ell}}h^{\ell}= c_{\ell}^{dom}(\bigoplus_{i=1}^c \mathcal{O}_{X}(d_i-m))  = c_{\ell}(\bigoplus_{i=1}^c \mathcal{O}_{X}(d_i)).
\end{eqnarray}

In particular, for $\ell \leqslant c$, we have $\deg(\tilde{s_\ell})=\ell$ and $\tilde{s_\ell}^{dom}>0$.
\begin{remark}
If $n\leqslant c$ we have that equality \ref{dominant} holds for all ${\ell}\in \mathbb{Z}$ (the case ${\ell}>n$ is obvious since both side of the equality vanish by a dimension argument).
\end{remark}

With this we can give some estimates for some intersection products on $X$.

\begin{lemma}\label{intersection}
\begin{enumerate}
\item Take $0\leqslant i_1\leqslant \cdots \leqslant i_k $, $\ell>0$ such that $i_1+\cdots +i_k+\ell=n$ then
$$\deg\left( \int_X s_{i_1}\cdots s_{i_k}h^{\ell} \right)<N$$ \label{intersection1}
\item Take $0\leqslant i_1\leqslant \cdots \leqslant i_k $ such that $i_1+\cdots +i_k=n$ then
 $\int_{X}s_{i_1}\cdots s_{i_k}$ is of degre $N$ if and only if $i_k\leqslant c$.

\item Take $0\leqslant i_1 \leqslant \cdots \leqslant i_{\kappa}$.
If $i_1< b$ or if $i_1 =b$ and $i_j< c$ for some $j>1$ then  $$\deg \left(\int_{X}s_{i_1}\cdots s_{i_{\kappa}}\right)<N.$$

\end{enumerate}
\end{lemma}


\begin{proof}
First note that for $0\leqslant {\ell} \leqslant n$, since $\deg \int_X h^n=c $,  
$$\deg\left( \int_X s_{i_1}\cdots s_{i_k}h^{\ell} \right)=\deg{\tilde{s}_{i_1}}+\cdots +\deg{\tilde{s}_{i_k}} +c.$$
Recall also that $\deg(\tilde{s_{i}})=i$ if $i\leqslant c$ and $\deg(\tilde{s_{i}})=c$ if $i > c$.
Now the first point is clear, let ${\ell} >0$
\begin{eqnarray}
\deg\left( \int_X s_{i_1}\cdots s_{i_k}h^{\ell} \right)&\leqslant & i_1 +\cdots + i_k +c <n+c=N \nonumber
\end{eqnarray}
To see the second point, if $i_{k}\leqslant c$ so are all the $i_j$ and therefore $$\deg \left(\int_{X} s_{i_1}\cdots s_{i_{k}}\right)=i_1 + \cdots +i_{k} +c = N.$$ Conversely if $i_{k}>c$ then $$\deg \left(\int_{X}s_{i_1}\cdots s_{i_{k}}\right)=i_1 + \cdots +i_{k-1}+c +c <N.$$
The last point is an easy consequence of the second one thanks to the equality $n=(\kappa -1)c +b$.
\end{proof}


\subsection{Nonvanishing for Jet Differentials}

We can now state and prove our non vanishing theorem which will show that Diverio's result (\cite{Div08} Theorem 7) is optimal. Recall that $\kappa$ denotes the ratio~$\lceil \frac{n}{c} \rceil$.

\begin{theorem}\label{thmjet}
Fix $a\geqslant 0$. There exists a constant $\Gamma_{N,n,a}$ such that, if for all $i$  $d_i\geqslant \Gamma_{N,n,a}$, then $\mathcal{O}_{X_\kappa}(1)\otimes \pi_\kappa^* \mathcal{O}_X(-a)$ is big on $X_{\kappa}$. In particular
$$H^0(X,E_{\kappa,m}\Omega_X(-ma))\neq 0$$
when $m\gg 0$.
\end{theorem} 


\begin{remark}
It will be clear during the proof that all computations could be made explicit, at least in small dimensions, so that the $\Gamma_{N,n,a}$ in the theorem is, in theory at least, effective. We will give an explicit bound in the case $\kappa=1$ in section \ref{effectifsymmetric} . 
\end{remark}


\begin{remark}
It is also natural to ask if under the hypothesis of the theorem, $\mathcal{O}_{X_k}(1)\otimes \pi_k^* \mathcal{O}_X(-a)$ is big if $k\geqslant \kappa$. In fact this is true and the proof is exactly the same.  
\end{remark}

As in $\cite{Div09}$ we introduce a nef line bundle $L_k\in \mathcal{C}_k(X)$ for all $k\geqslant 1$.

$$L_k:=\mathcal{O}_{X_k}(2\cdot 3^{k-2},\cdots , 2\cdot 3,2,1)\otimes\pi^*
\mathcal{O}_X(2\cdot 3^{k-1})$$
We can write its first chern class
$${\ell}_k:=c_1(L_k)=u_k +\beta_k$$
where $\beta_k$ is a class that comes from $X_{k-1}$.
Now we can state the main technical lemma. This is just the combination of
lemmas $\ref{labase}$ and $\ref{intersection}$.


\begin{lemma}\label{technical}
With the above notations we have the following estimates.
\begin{enumerate}
\item Take $k\geqslant 1$ and $\gamma_1,\cdots ,\gamma_{n_k-1} \in \mathcal{C}_{k}(X)$. Then 
$$\int_{X_k} \gamma_1 \cdots \gamma_m h=o(d^N).$$
\item Take $\gamma_1,\cdots ,\gamma_p \in \mathcal{C}_{k}(X)$ and $0\leqslant i_1 \leqslant \cdots \leqslant i_{q}$ such that $p+\sum i_j=n_k$. If $i_1< b$ or if $i_1 =b$ and $i_j< c$ for some $j>1$ then we have the following estimates, 
\begin{eqnarray}
\int_{X_k}s_{k,i_1}\cdots s_{k,i_q}\gamma_1\cdots \gamma_p&=&o(d^N)\label{estim1}\\
\int_{X_{k}}s_{k-1,i_1}\cdots s_{k-1,i_q}\gamma_1\cdots \gamma_p&=&o(d^N)\label{estim}
\end{eqnarray}
\item Take $0<k<\kappa$. Then
$$\int_{X_k} s_{k,b}s_{k,c}^{\kappa-k-1}{\ell}_{k}^{\hat{c}}\cdots {\ell}_1^{\hat{c}}= \int_{X_{k-1}} s_{k-1,b}s_{k-1,c}^{\kappa-k}{\ell}_{k-1}^{\hat{c}}\cdots {\ell}_1^{\hat{c}}+o(d^N).$$

\end{enumerate}
\end{lemma}


\begin{proof}

1) is an immediate consequence of Lemma $\ref{intersection}.1$ and Lemma $\ref{labase}$.
Similarly for 2), thanks to Lemma $\ref{labase}$ we write
$$\int_{X_k}s_{k,i_1}\cdots s_{k,i_{q}}\gamma_1\cdots \gamma_p= \sum_{j_1,\cdots,j_{k+q}} Q_{j_1,\cdots,j_{k+q}}\int_X s_{j_1}\cdots s_{j_{k+q}} h^{a}$$

where $a\geqslant 0$ and moreover we know that in each term of this sum either $j_1<b$ or $j_p< c$ for some $p>0$, thus we can apply Lemma $\ref{intersection}.1$ (if $a>0$) or $\ref{intersection}.2$ (if $a=0$). From this one can easily deduce formula \ref{estim}, write $\gamma_i=a_iu_i+\beta_i$ where $\beta_i\in \mathcal{C}_{k-1}(X)$. 
\begin{eqnarray}
\int_{X_{k}}s_{k-1,i_1}\cdots \!\!\!\! &s_{k-1,i_q}&\!\!\!\!\gamma_1\cdots \gamma_p = \int_{X_{k}}s_{k-1,i_1}\cdots s_{k-1,i_q}(a_1u_1+\beta_1)\cdots (a_1u_1+\beta_p)\nonumber \\
&=&\!\!\!\!\!\!\!\!\!\! \sum_{I\subseteq{\{1,\cdots,p \}}}{\left( \prod_{i\in I}{a_i} \right)}\int_{X_{k}}s_{k-1,i_1}\cdots s_{k-1,i_q}u_k^{|I|}\prod_{i\notin I}{\beta_i}\nonumber \\
&=&\!\!\!\!\!\!\!\!\!\! \sum_{I\subseteq{\{1,\cdots,p \}}}{\left( \prod_{i\in I}{a_i} \right)}\int_{X_{k-1}}\!\!\!\!\!s_{k-1,i_1}\cdots s_{k-1,i_q}\cdot s_{k-1,|I|-n+1}\prod_{i\notin I}{\beta_i}\nonumber
\end{eqnarray}
Then we conclude by applying formula $\ref{estim1}$.

To see 3), 
\begin{eqnarray}
\int_{X_k} \!\!\!s_{k,b}s_{k,c}^{\kappa - k -1} {\ell}_{k}^{\hat{c}}\cdots {\ell}_{1}^{\hat{c}}\!\!\!\!&=&\!\!\!\! \int_{X_k}\!\! \left(\sum_{i=0}^{b}M_{b,i}^{n} s_{k-1,i}u_{k}^{b-i} \right)\!\! \left( \sum_{i=0}^{c}M_{c,i}^{n} s_{k-1,i}u_{k}^{c-i} \right)^{\kappa-k-1}\!\!\!\!\!\!\!\!\!\!\!\!\!\!\!\!\! {\ell}_{k}^{\hat{c}}\cdots {\ell}_{1}^{\hat{c}}\nonumber \\ 
                 &=& \!\!\!\! \int_{X_k}  s_{k-1,b}s_{k-1,c}^{\kappa - k -1} {\ell}_{k}^{\hat{c}}\cdots {\ell}_{1}^{\hat{c}}+o(d^N)\nonumber
\end{eqnarray}
By expanding and using formula \ref{estim} in each term of the obtained sum. Now,
\begin{eqnarray}
\int_{X_k}  s_{k-1,b}s_{k-1,c}^{\kappa - k -1} {\ell}_{k}^{\hat{c}}\cdots {\ell}_{1}^{\hat{c}}+o(d^N) \!\!\!\!&=&\!\!\!\! \int_{X_k} \!\!\!\! s_{k-1,b}s_{k-1,c}^{\kappa - k -1} (u_{k}+\beta_k)^{\hat{c}}{\ell}_{k-1}^{\hat{c}}\cdots {\ell}_{1}^{\hat{c}}\nonumber \\
                 &=&\!\!\!\! \int_{X_k} \!\!\!\! s_{k-1,b}s_{k-1,c}^{\kappa - k -1}  \sum_{i=0}^{\hat{c}}\binom{\hat{c}}{i}u_k^{\hat{c}-i}\beta_k^i {\ell}_{k-1}^{\hat{c}}\cdots {\ell}_{1}^{\hat{c}}\nonumber \\
                 &=&\!\!\!\! \int_{X_{k-1}} \!\!\!\!\!\!\! s_{k-1,b}s_{k-1,c}^{\kappa - k -1}  \sum_{i=0}^{c}\binom{\hat{c}}{i}s_{k-1,i}\beta_k^i {\ell}_{k-1}^{\hat{c}}\cdots {\ell}_{1}^{\hat{c}}\nonumber \\
                 &=&\!\!\!\! \int_{X_{k-1}}  \!\!\!\!\!\!\!s_{k-1,b}s_{k-1,c}^{\kappa - k -1} s_{k-1,c}{\ell}_{k-1}^{\hat{c}}\cdots {\ell}_{1}^{\hat{c}} +o(d^N)\nonumber \\
                 &=&\!\!\!\! \int_{X_{k-1}}  \!\!\!\!\!\!\!s_{k-1,b}s_{k-1,c}^{\kappa - (k-1) -1} {\ell}_{k-1}^{\hat{c}}\cdots {\ell}_{1}^{\hat{c}} +o(d^N)\nonumber
\end{eqnarray}

\end{proof}


Recall also the following consequence of holomorphic Morse inequalities (see $\cite{Laz04}$).


\begin{theorem}
Let $Y$ be a smooth projective variety of dimension $n$ and let $F$ and $G$ be two nef divisors on $Y$. If $F^n>nG\cdot F^{n-1}$ then $F-G$ is big.
\end{theorem}


We are now ready to prove Theorem $\ref{thmjet}$. 


\begin{proof}
First we recall an argument from $\cite{Div09}$ to show that we just have to check that $\mathcal{O}_{X_k}(a_1,\cdots, a_k)\otimes \pi_k^*\mathcal{O}_X(-a)$ is big for some suitable $a_i$'s. We know (see $\cite{Dem97}$) that $D_k:=\mathbb{P}(\Omega_{X_{k-1}/X_{k-2}})\subset X_k$ is an effective divisor that satisfies the relation $\pi_{k-1,k}^*\mathcal{O}_{X_{k-1}}(1)=\mathcal{O}_{X_k}(1)\otimes \mathcal{O}_{X_k}(-D_k)$. From this, an immediate induction shows that for any $k>1$ and any $k$-uple $(a_1,\cdots,a_k)$ we have 
$$\mathcal{O}_{X_k}(b_{k+1})=\mathcal{O}_{X_k}(a_1,\cdots, a_k)\otimes\pi_{2,k}^*\mathcal{O}_{X_2}(b_1D_2)\otimes\cdots \otimes\mathcal{O}_{X_k}(b_{k-1}D_k).$$
Where for all $j>0$, $b_j:=a_1+\cdots +a_{j}$. Thus when $0\leqslant b_j$ for all $0\leqslant j\leqslant k$ then $\pi_{2,k}^*\mathcal{O}_{X_2}(b_1D_2)\otimes\cdots  \otimes\mathcal{O}_{X_k}(b_{k-1}D_k)$ is effective, this means that, under this condition, to prove that $\mathcal{O}_{X_k}(1)\otimes \pi_k^*\mathcal{O}_X(-a)$ is big it is sufficient to show that $\mathcal{O}_{X_k}(a_1,\cdots, a_k)\otimes \pi_k^*\mathcal{O}_X(-a)$ is big.

 Let $D=F-G$ where, as in \cite{Mou10}, we set $F:=L_{\kappa} \otimes \cdots \otimes L_1$, and $G=\pi_{\kappa}^*(\mathcal{O}_X(m+a))$ where $m\geqslant 0$ is chosen so that $F\otimes\pi_{\kappa}\mathcal{O}_X(-m)$ has no component coming from $X$. It is therefore sufficent to show that $D$ is big. To do so, we will apply holomorphic Morse inequalities to $F$ and $G$ (both nef).
We need to prove that 
$$F^{n_{\kappa}}>n_{\kappa}F^{n_{\kappa}-1}\cdot G.$$
Clearly, the right hand side has degree strictly less than $N$ in the $d_i$'s thanks to Lemma $\ref{technical}$
and therefore we just have to show that the left hand side is larger than a positive polynomial of degree $N$ in the $d_i$'s. Let $\alpha := c_1(\pi_{\kappa}^*(\mathcal{O}_X(a))$

\begin{eqnarray}
F^{n_{\kappa}}&=& \int_{X_\kappa}({\ell}_{\kappa}+\cdots +{\ell}_1-\alpha)^{n_{\kappa}} \nonumber \\
				 &=& \int_{X_\kappa}\sum_{i=0}^{n_{\kappa}}(-1)^i\binom{n_{\kappa}}{i} ({\ell}_{\kappa}+\cdots +{\ell}_1)^{n_{\kappa}-i}\alpha^i \nonumber \\
                 &=& \int_{X_\kappa}({\ell}_{\kappa}+\cdots +{\ell}_1)^{n_{\kappa}}+o(d^N)\nonumber
\end{eqnarray}
by applying Lemma \ref{technical}. But since all the $\ell_i$'s are nef,
\begin{eqnarray}
\int_{X_\kappa}({\ell}_{\kappa}+\cdots +{\ell}_1)^{n_{\kappa}}&\geqslant & \int_{X_\kappa} {\ell}_{\kappa}^{\hat{b}} \cdot {\ell}_{\kappa-1}^{\hat{c}}\cdots {\ell}_1^{\hat{c}}\nonumber \\
                 &=& \int_{X_{\kappa -1}}\!\!\!s_{\kappa -1,b} \cdot {\ell}_{\kappa-1}^{\hat{c}}\cdots {\ell}_1^{\hat{c}}+o(d^N)\nonumber
\end{eqnarray}
The last inequality is obtained by using Lemma \ref{technical}.2. Now an immediate induction proves that for all $k<\kappa$ one has
$$F^{n_{\kappa}} \geqslant \int_{X_k} s_{k,b}s_{k,c}^{\kappa-k-1} l_{k}^{\hat{c}}\cdots l_{1}^{\hat{c}}+o(d^N).$$

We just proved the case $k = \kappa -1$ and the other part of the induction is exactly the content of Lemma $\ref{technical}.3$.
Therefore,
$$F^{n_{\kappa}} \geqslant \int_{X}  s_bs_c^{\kappa-1}+o(d^N)$$
and we conclude applying Lemma $\ref{intersection}.2$
\end{proof}


%
%

\section{Positivity of the Cotangent Bundle}\label{positivitecotangent}

From now on we will focus on the positivity of the cotangent bundle. the organisation is the following. 
In \ref{Omega2} we give a geometric interpretation of the ampleness of $\Omega_X\otimes \cO_{\bP^N_{\vert X}}(2)$.
In \ref{numericalpositivity} we prove that all positivity conditions on the Chern classes that one might expect are indeed satisfied, this proves the numerical side of the conjecture.
Then in \ref{mainresult} we state and prove our main result, admitting, temporarily,  a global generation property whose proof is delayed to section  \ref{globallygenerated}. Then, in \ref{effectifsymmetric}, we will give an effective bound on the degree of the hypersurfaces we are intersecting to ensure the existence of global twisted symmetric differential forms on our complete intersection variety.  



\subsection{Ampleness of $\Omega_X(2)$}\label{Omega2}

The first remark to make is that the cotangent bundle to an abelian variety is nef, but the cotangent bundle to $\bP^N$ isn't, this is the reason why Debarre's Conjecture should be more difficult to solve than the analoguous statement in an abelian variety. But we know that $\Om_X(2)$ is nef since it is globally generated, therefore it seems a natural question to see for which variety this bundle is ample. It turns out that this has a neat geometric interpretation, which is that $\Om_X(2)$ is ample if and only if there is no line in $X$, as we shall now see.\\
Fix a $N+1$-dimensional complex vector space $V$, denote $\mathbb{P}^N=\mathbb{P}(V^*)$ the projectivized space of lines in $V$, $p : V \to \mathbb{P}(V^*)$ the projection, and $Gr(2,V)=Gr(2,\mathbb{P}^N)$ the space of vector planes in $V$ which is also the space of lines in $\mathbb{P}^N$. We will also consider the projection $\pi : \mathbb{P}(\Omega_{\mathbb{P}^N})\to \mathbb{P}^N$. 
The key point is the following lemma.


\begin{lemma}\label{Han}
There is a map $\varphi: \mathbb{P}(\Omega_{\mathbb{P}^N})\to \mathbb{P}(\Lambda^2 V^*)$ such that $\varphi^* \mathcal{O}_{\mathbb{P}(\Lambda^2 V^*)}(1)=\mathcal{O}_{\Omega_{\mathbb{P}^N}}(1)\otimes \pi^*\mathcal{O}_{\mathbb{P}^N}(2)$. Moreover this application factors through the Pl\"{u}cker embedding $Gr(2,\mathbb{P}^N)=Gr(2,V)\hookrightarrow \mathbb{P}(\Lambda^2 V^*)$. More precisely, an element $(x,[\xi])\in \mathbb{P}(\Omega_{\mathbb{P}^N})$ with $x\in\mathbb{P}^N$ and $\xi \in T_x \mathbb{P}^N$, gets mapped to the unique line $\Delta$ in $\mathbb{P}^N$ satisfying $\xi\in T_x\Delta \subseteq T_x\mathbb{P}^N$. 
\end{lemma}


\begin{proof} 
Take the Euler exact sequence 
$$0 \to \mathcal{O}_{\mathbb{P}^N} \to V\otimes \mathcal{O}_{\mathbb {P}^N}(1)\to T\mathbb{P}^N \to 0$$ 
and apply $\Lambda^{N-1}$ to it in order to get the quotient
$$\Lambda^{N-1}V\otimes\mathcal{O}_{\mathbb{P}^N}(N-1)\to \Lambda^{N-1}T \mathbb{P}^N \to 0.$$
Now using the well known dualities, $\Lambda^{N-1}V=\Lambda^2V^*$ and $\Lambda^{N-1}T \mathbb{P}^N=\Omega_{\mathbb{P}^N}\otimes K_{\mathbb{P}^N}^*=\Omega_{\mathbb{P}^N}\otimes\mathcal{O}_{\mathbb{P}^N}(N+1)$, and tensoring everything by $\mathcal{O}_{\mathbb{P}^N}(1-N)$ we get 
$$\Lambda^2V^*\to \Omega_{\mathbb{P}^N}\otimes \mathcal{O}_{\mathbb{P}^N}(2)\to 0.$$
This yields the map $\varphi : {\mathbb{P}}(\Omega_{\mathbb{P}^N})=\mathbb{P}(\Omega_{\mathbb{P}^N}(2))\hookrightarrow {\mathbb{P}^N}\times \mathbb{P}(\Lambda^2V^*)\to \mathbb{P}(\Lambda^2V^*)$ such that $\varphi^*\mathcal{O}_{\mathbb{P}(\Lambda^2 V^2)}(1)=\mathcal{O}_{\Omega_{\mathbb{P}^N}(2)}(1)=\mathcal{O}_{\Omega_{\mathbb{P}^N}}(1)\otimes \pi^*\mathcal{O}_{\mathbb{P}^N}(2)$.

To see the geometric interpretation of this map, it suffices to backtrack through the previous maps. Take a point $x\in \mathbb{P}^N$ and a vector $0 \neq \xi \in T_x \mathbb{P}^N$ and fix a basis $(\xi_0,\cdots, \xi_{N-1})$ of $T_x\mathbb{P}^N$ such that $\xi_0 = \xi$. Now take $v\in V$ such that $p(v)=x$ and a basis $(e_0,\cdots, e_N)$ of $T_vV= V$ such that $d_vp(e_N)=0$ and $d_vp(e_i)=\xi_i$ for $i<N$. We just have to check that $(x,\left[\xi\right])$ is mapped to the annouced line $\Delta$ which, with our notations, corresponds to the point $\left[e_0\wedge e_N\right]\in \mathbb{P}(\Lambda^2V^*)$.\\ 

$\begin{array}{ccccccc}
\mathbb{P}(\Omega_{\mathbb{P}^N,x}) &\to& \mathbb{P}(\Lambda^{N-1}T_x\mathbb{P}^N) &\to& \mathbb{P}(\Lambda^{N-1}V) &\to& \mathbb{P}(\Lambda^2V^*)\\
\left[ \xi_0 \right]  &\mapsto& \left[ \xi_1^*\wedge \cdots \wedge \xi_{N-1}^* \right] &\mapsto& \left[ e_1^*\wedge \cdots \wedge e_{N-1}^* \right] &\mapsto& \left[ e_0\wedge e_N \right]
\end{array}$
\end{proof}


With this we can prove our proposition,


\begin{e-proposition}\label{droite}
Let $X\subseteq \mathbb{P}^N$ be a smooth variety then $\Omega_X(2)$ is ample if and only if $X$ doesn't contain any line.
\end{e-proposition} 


\begin{proof}
By Lemma~$\ref{Han}$, we know that $\Omega_X(2)$ is ample if and only if the restriction of $\varphi$, $\varphi_{X}: \mathbb{P}(\Omega_X)\subseteq \mathbb{P}(\Omega_{\mathbb{P}^N})\to Gr(2,\mathbb{P}^N)$ is finite.

Now, if $X$ contains a line $\Delta$ then $\varphi_X$ is not finite since the curve $\mathbb{P}(K_{\Delta})~\subseteq~\mathbb{P}(\Omega_X)$ gets mapped to the point in $Gr(2,\mathbb{P}^N)$ representing $\Delta$.

If $\varphi_{X}$ is not finite then there is a curve $C\subseteq \mathbb{P}(\Omega_X)$ which gets mapped to a point in $Gr(2,\mathbb{P}^N)$ corresponding to a line $\Delta$ in $\mathbb{P}^N$. Let $\Gamma=\pi(C)$, Lemma~$\ref{Han}$ tells us that the embedded tangent space $\mathbb{T}_x\Gamma$ equals $\Delta$ for all $x\in \Gamma$ and therefore  $\Delta\subseteq X$.   
\end{proof}


Note that by a dimension argument on the incidence variety of lines on the universal hypersurface one can see that a generic hypersurface $H$ of high degree contains no line and therefore $\Omega_H\otimes \mathcal{O}_{\mathbb{P}^N}(2)$ is ample. 



\subsection{Numerical positivity of the cotangent bundle}\label{numericalpositivity}

As an other evidence towards this conjecture we will now check that all the positivity conditions on the chern classes that one might expect (according to a theorem of Fulton and Lazarsfeld \cite{F-L83}) are indeed satisfied. We start by a preliminary section on numerical positivity before turning to the proof of our result.

\subsubsection{Numerical positivity}
Following Fulton $\cite{Ful98}$ we recall definitions concerning Schur polynomials.
Let $c_1, c_2, c_3,...$ be a sequence of formal variables.
Let $\ell$ be a positive integer and $\lambda=(\lambda_1,...,\lambda_{\ell})$ be a partition of ${\ell}$. We define the Schur polynomial associated to $c=(c_i)_{{i\in \mathbb{N}}}$ and $\lambda$ to be:
\begin{eqnarray}
\Delta_{\lambda}(c):= \det \left[(c_{\lambda_i+j-i})_{1\leqslant i,j\leqslant {\ell}}\right]\nonumber
\end{eqnarray}
For exemple, $\Delta_{1}(c)=c_1$, $\Delta_{2,0}(c)=c_2$ and $\Delta_{(1,1)}(c)=c_1^2-c_2$.

Now consider two sequences of formal variables, $c_1, c_2, c_3,...$ and $s_1, s_2, s_3,...$ satisfying the relation:
\begin{eqnarray}
(1+c_1 t + c_2 t^2+ ...)\cdot(1-s_1 t +s_2 t^2 - ...)=1 \label{chernsegre}
\end{eqnarray}

Note that relation (\ref{chernsegre}) is satisfied when $c_i=c_i(E)$ are the Chern classes of a vector bundle $E$ over a variety $X$ and $s_i=s_i(E)$ are its Segre classes.

The proof of the following crucial combinatorial result can be found in $\cite{Ful98}$.


\begin{lem}\label{lemmecombi}
\emph{(\cite{Ful98})}
With the same notations.
Let $\bar{\lambda}$ be the conjugate partition of $\lambda$, then $\Delta_{\lambda}(c)=\Delta_{\bar{\lambda}}(s)$.
\end{lem}

Let $E$ be a vector bundle of rank $r$  over a projective variety $X$ of dimension~$n$. 

\begin{e-definition}
We will say that $E$ is numerically positive if for any subvariety $Y \subseteq X $ and for any partition $\lambda $ of $ {\ell}=dim(Y) $ one has $\int_{Y}\Delta_{\lambda}(c(E))>0$.
\end{e-definition}

This definition is motivated by a theorem of Fulton and Lazarsfeld $\cite{F-L83}$ which gives numerical consequences of ampleness.
\begin{thm}\emph{(\cite{F-L83})}
If $E$ is ample then $E$ is numerically positive. Moreover the Schur polynomials are exactly the relevant polynomials to test ampleness numerically.
\end{thm}
Note that the converse is false, for example the bundle $\mathcal{O}_{\mathbb{P}^1}(2)\oplus \mathcal{O}_{\mathbb{P}^1}(-1)$ over $\mathbb{P}^1$ is numerically positive but not ample (the problem being the lack of subvarieties of dimension two to test $c_2$). See $\cite{Ful76}$ for a more interesting example.

\subsubsection{Numerical Positivity of the Cotangent Bundle } 

We can now prove that the cotangent bundle to a complete intersection variety is numerically positive if the codimension is greater than the dimension and if the multidegree is big enough. We will use the notations of section \ref{sectionsegre}.

\begin{theorem}\label{numpos}
Fix $a\in\mathbb{Z}$.  There exists $D_{N,n,a}\in \mathbb{N}$ such that if $X\subset \bP^N$ is a complete intersection of dimension $n$, of codimension $c$ and multidegree $(d_1,\cdots, d_c)$ such that $c\geqslant n$ and $d_i > D_{N,n,a}$ for all $i$ then $\Omega_X(-a)$ is numerically positive.
\end{theorem}   


\begin{proof}
By Lemma \ref{lemmecombi} we have to check that for any subvariety $Y\subseteq X$ of dimension $\ell\leqslant n\leqslant c$ and for any partition $\lambda$ of $\ell$ one has $\int_Y \Delta_{\bar{\lambda}}(s(\Omega_X(-a)))>0$. Moreover, $\int_Y \Delta_{\bar{\lambda}}(s(\Omega_X(-a)))= \tilde{\Delta}_{\bar{\lambda}}(s(\Omega_X(-a)))\int_Yh^{\ell}$, thus we just have to check that $\tilde{\Delta}_{\bar{\lambda}}(s(\Omega_X(-a)))>0$ when the $d_i$'s are large enough, which is equivalent to $\tilde{\Delta}^{dom}_{\bar{\lambda}}(s(\Omega_X(-a)))>0$. Now the equality 
\begin{eqnarray}
 \tilde{\Delta}_{\bar{\lambda}}^{dom}(s(\Omega_X(-a)))=\det (\tilde{s}_{\lambda_i+j-i}^{dom}(\Omega_X(-a)))_{1\leqslant i,j\leqslant \ell}\label{equa}
\end{eqnarray}
holds if one can prove that the right hand side is non zero.
But by (\ref{dominant}) we find :
\begin{eqnarray}
\det (\tilde{s}_{\lambda_i+j-i}^{dom}(\Omega_X(-a)))_{1\leqslant i,j\leqslant l}&=&\det \left(\tilde{c}_{\lambda_i+j-i}\left(\bigoplus_{j=1}^k\mathcal{O}(d_j)\right)\right)_{1\leqslant i,j\leqslant l}\nonumber \\
&=&\tilde{\Delta}_{\bar{\lambda}}\left(c\left(\bigoplus_{j=1}^k\mathcal{O}(d_j)\right)\right)\nonumber
\end{eqnarray}
By applying the theorem of Fulton and Lazarsfeld to $\bigoplus_{j=1}^k\mathcal{O}_X(d_j)$ (which is ample if $d_i>0$) we find that this is strictly positive. 
This yields equality in \ref{equa}, and we get the desired result.
\end{proof}

\begin{remark}
Note that there is no assumption made on the genericity of our complete intersection.
\end{remark}



\subsection{Main Result}\label{mainresult}

We will now state our main results.

\begin{theorem}\label{main}
Fix $a\in\bN$. There exists $\delta_{N,n,a} \in \bN$ such that, if $X\subset \bP^N$ is a generic complete intersection of dimension $n$, codimension $c$ and multidegree $(d_1,\cdots, d_c)$ satisfying $c\geq n$ and $d_i\geq \delta_N$ for all $1\leq i \leq c$, then $\cO_{\bP(\Om_X)}(1)\otimes\pi_{\Om_X}^{*}\cO_X(-a)$ is big and there exists a subset $Y\subset X$ of codimension at least two such that,
$$\pi_{\Om_X}(\Bs(\cO_{\bP(\Om_X)}(1)\otimes\pi_{\Om_X}^{*}\cO_X(-a))) \subseteq Y.$$

\end{theorem}

\begin{remark}
We will give an effective bound for $\delta_{N,n,a}$ in section \ref{effectifsymmetric}
\end{remark}

From this we are to deduce some other noteworthy conclusions. First, almost everywhere ampleness in the sense of Miyaoka \cite{Miy83}.

\begin{e-definition}
Let $E$ be a vector bundle on a variety $X$ and $H$ an ample line bundle on $X$. Denote the projection by $\pi : \bP (E)\to X$. Take $T \subset X$. We say that $E$ is ample modulo $T$ if for sufficiently small $\epsilon > 0$, any irreducible curve $C\subset \bP (E)$ such that $\int_C c_1(\cO_{\bP(E)}(1))<\epsilon C\cdot \pi^*H$ satisfies $\pi (C)\subseteq T$. We say that $E$ is almost everywhere ample if there is a proper closed algebraic subset $T\subset X$ such that $E$ is ample modulo $T$. 
\end{e-definition}

\begin{corollary}
Fix $a\in \bN$. If $X\subseteq \bP^N$ is a complete intersection variety of dimension $n$, codimension $c$ and multidegree $(d_1,\cdots , d_c)$ satisfying $c\geq n$ and $d_i\geq \delta_{N,n,(a+1)}$, then $\Om_X\otimes \cO_X(-a)$ is ample outside an algebraic subset of codimension at least two in $X$   
\end{corollary}

\begin{proof}
Applying the theorem we find that  $\cO_{\bP(\Om_X)}(1)\otimes\pi_{\Om_X}^{*}\cO_X(-a-1)$ is big an that there exists an algebraic subset $Y\subset X$ of codimension 2 in $X$ such that $\pi_{\Om_X}(\Bs(\cO_{\bP(\Om_X)}(1)\otimes\pi_{\Om_X}^{*}\cO_X(-a-1))) \subseteq Y$. Now take any $\epsilon \leq 1$. Take an irreducible curve $C\subset \bP(\Om_X(-a))=\bP(\Om_X)$ such that $$\int_C\cO_{\bP(\Om_X(-a))}(1)=\int_C\cO_{\bP(\Om_X)}(1)\otimes \pi_{\Om_X}^*\cO_X(-a)<\epsilon \int_C\pi_{\Om_X}*\cO_X(1)\leq \int_C\pi_{\Om_X}^*\cO_X(1).$$

This implies that $\int_C\cO_{\bP(\Om_X)}(1)\otimes \pi_{\Om_X}^*\cO_X(-a-1)<0$. Thus in particular we get $C\subseteq \Bs(\cO_{\bP(\Om_X)}(1)\otimes\pi_{\Om_X}^{*}\cO_X(-a-1))$ and thus $\pi_{\Om_X} (C)\subseteq Y$.
\end{proof}

And we also have the announced positive answer to Debarre's conjecture for surfaces (the bound given here will be discussed in section \ref{dim2} ).

\begin{corollary}\label{surface}
If $N\geq 4$ and $S\subset \bP^N$ a generic complete intersection surface of multidegree $(d_1,\cdots,d_{N-2})$ satisfying $d_i\geq \frac{8N+2}{N-3}$, then $\Om_S$ is ample.
\end{corollary}

In particular, in $\bP^4$, generically, the intersection of two hypersurfaces of degree greater than $34$ have ample cotangent bundle. To our knowledge this is the first example of surfaces with ample cotangent bundle in $\bP^4$, a question that was already raised by M. Schneider in \cite{Sch92}.


\subsubsection{Proof of the main theorem}\label{mainproof}
First let us introduce some notations let $\textbf{P}:=\bP^{N_{d_1}}\times \cdots \times \bP^{N_{d_c}}$ where $\bP^{N_{d_i}}:=\bP(H^0(\bP^N,\cO_{\bP^N}(d_i))^*)$, and $\cX:=\{(x,(t_1,\cdots,t_c))\in \bP^N\times\textbf{P}\ \ /\ \ t_i(x)=0 \ \ \forall i \}$. We will also denote $\rho_1:\cX \to \bP^N$ the projection onto the first factor and $\rho_2:\cX\to \textbf{P}$. We will use the standard notation $\cO_{\textbf{P}}(a_1,\dots,a_c)$ to denote line bundles on $\textbf{P}$. Also, write $\pi: \bP(\Om_{\cX / \gP})\to \cX$ for the standard projection.
As in \cite{DMR10} the proof of this theorem is based on Theorem \ref{thmjet} and on a global generation property that we will prove in \ref{globallygenerated}.

\begin{thmglobal}
The bundle
$$T\bP(\Om_{\cX / \gP})\otimes \pi^*\rho_{1}^*\cO_{\bP^N}(N)\otimes \pi^*\rho_{2}^*\cO_{\gP}(1,\cdots,1)$$
is globally generated on $\bP(\Om_{\cX / \gP})$.
\end{thmglobal} 

Observe also the following simple remark.

\begin{remark}\label{base}
Let $X\subset \bP^N$ be any projective variety. If $q\geq0$ then $$\Bs(\cO_{\bP(\Om_X)}(1)\otimes\pi_{\Om_X}^{*}\cO_X(-a))\subseteq \Bs(\cO_{\bP(\Om_X)}(1)\otimes\pi_{\Om_X}^{*}\cO_X(-a-q)).$$
\end{remark}

We are now in position to give the proof of the main result. We take the notation of Theorem \ref{main}. We will prove that $\delta_{N,n,a}:=\Gamma_{N,n,(a+N)}$ (with the notation of Theorem \ref{thmjet}) will suffice. Fix a multidegree $(d_1,\cdots,d_c)$ such that $d_i\geq \delta_{N,n,a}=\Gamma_{N,n,(a+N)}$ for all $0\leq i\leq c$.  We start by applying Theorem \ref{thmjet} to find some $k\in \bN$ such that $H^0(X,S^k\Om_X\otimes\cO_X(-ka-kN))\neq 0$. Now take $U\subset \gP$ such that the restriction map 
$$H^0(\cX_U,S^k\Om_{\cX / \gP}\otimes\rho_1^*\cO_{\bP^N}(-ka-kN))\to H^0(X_t,S^k\Om_{X_t}\otimes \cO_{X_t}(-ka-kN))$$
is surjective for all $t\in U$. Where $\cX_U:=\rho^{-1}_{2}(U)$ and $X_t:=\rho^{-1}_{2}(\{t\})$. Fix $t_0\in U$ and a non zero section $$\sigma_0\in H^0(X_{t_0},S^k\Om_{X_{t_0}}\otimes\cO_{X_{t_0}}(-ka-kN))$$ and extend it to a section $$\sigma\in H^0(\cX_U,S^k\Om_{\cX / \gP}\otimes\rho_1^*\cO_{\bP^N}(-ka-kN)).$$
 Let $\cY:=(\sigma=0)\subset \cX_U$. We will prove that $$\pi_{t_0}(\Bs(\cO_{\bP(\Om_{X_{t_0}})}(1)\otimes\pi_{\Om_{X_{t_0}}}^{*}\cO_{X_{t_0}}(-a-q)))\subset Y_0.$$ 
Where $Y_t:=X_t\cap \cY$. Denote $\tilde{\sigma}$ the section in $H^0(\bP(\Om_{\cX/\gP})_{|U},\cO_{\bP(\Om_{\cX/\gP})}(k)\otimes\pi^*\rho_1^*\cO(-ka-kN))$ corresponding to $\sigma$ under the canonical isomorphism 
$$H^0(\bP(\Om_{\cX/\gP})_{|U},\cO_{\bP(\Om_{\cX/\gP})}(k)\otimes\pi^*\rho_1^*\cO_{\bP^N}(-ka-kN))=H^0(\cX_U,S^k\Om_{\cX / \gP}\otimes\rho_1^*\cO_{\bP^N}(-ka-kN)).$$

Let $x\in \Bs(\cO_{\bP(\Om_{X_{t_0}})}(1)\otimes\pi_{\Om_{X_{t_0}}}^{*}\cO_{X_{t_0}}(-a)) $ so that in particular, thanks to remark \ref{base},  $x\in \tilde{\sigma}_{t_0}$. We will now show that $\pi_{t_0}(x)\in Y_{t_0}$. Take coordinates around $x$ of the form $(t,z_i,[z'_i])$ such that $(t_0,0,[1:0:\cdots:0])=x$. In those coordinates we write 
\begin{eqnarray}
\sigma=\sum_{i_1+\cdots +i_n=k}{q_{i_1,\cdots,i_n}(t,z){z'_1}^{i_1}\cdots{z'_n}^{i_n}}.\nonumber
\end{eqnarray}
Therefore,
$$(\sigma=0)=\{(t,z)\ \ / \ \ q_{i_1,\cdots,i_n}(t,z)=0 \ \ \forall (i_1,\cdots,i_n)\}.$$

Fix any $(i_1,\cdots,i_n)\in \bN^n $ such that $i_1 + \cdots +i_n=k$ , we have to show that $q_{i_1,\cdots,i_n}(t_0,0)=0$. To do so, we apply Theorem \ref{global} to construct for each $1\leq j \leq n$

$$V_j \in H^0(\bP(\Om_{\cX / \gP}),T\bP(\Om_{\cX / \gP})\otimes\cO_{\bP^N}(N)\otimes\cO_{\gP}(1,\cdots,1))$$
 
such that in our coordinates, $V_j(t_0,0,[1:0:\cdots : 0])=\frac{\partial }{\partial z'_j}$.
Therefore by differentiating and contracting by $V_j$ $i_j$ times for each $1\leq j \leq n$ we get a new section
$$L_{V_1}\cdots L_{V_1}L_{V_2}\cdots L_{V_n}\tilde{\sigma}\in H^0(\bP(\Om_{\cX/\gP})_{|U},\cO_{\bP(\Om_{\cX/\gP})}(k)\otimes\pi^*\rho_1^*\cO_{\bP^N}(-ka)).$$

A local computation gives :
$$L_{V_1}\cdots L_{V_1}L_{V_2}\cdots L_{V_n}\tilde{\sigma}(t_0,0,[1:0:\cdots:0])=i_1!\cdots i_n!q_{i_1,\cdots,i_n}(t_0,0).$$
 
But since by hypothesis $x\in \Bs(\cO_{\bP(\Om_{X_{t_0}})}(1)\otimes\pi_{\Om_{X_{t_0}}}^{*}\cO_{X_{t_0}}(-a))$ we know that $$L_{V_1}\cdots L_{V_1}L_{V_2}\cdots L_{V_n}\tilde{\sigma}(t_0,0,[1:0:\cdots:0])=0.$$

Therefore we see that $q_{i_1,\cdots,i_n}(t_0,0)=0$, proving our claim.\\
To complete the proof we just have to show, as in \cite{D-T09}, the codimension two rafinement. Suppose that $Y_{t_0}$ has a divisorial component $E$. Since $E$ is effective and $Pic(X)=\mathbb{Z}$ we can deduce that $E$ is ample. Therefore there is an $m\in \mathbb{N}$ such that $mE$ is very ample. Now take $\sigma_{t_0}^m\in H^0(X,S^{km}\Om_{X_{t_0}}\otimes \cO_{X_{t_0}}(-mka-mkN))$, the divisorial component of the zero locus of $\sigma_{t_0}^m$ is $mE$. Now, for any $D \in |mE|$ we get a new section $\sigma_{t_0}^m\otimes D\otimes mE^{-1}\in H^0(X,S^{km}\Om_{X_{t_0}}\otimes \cO_{X_{t_0}}(-mka-mkN))$. By applying the same argument as above we know that the image of the base locus $\Bs(\cO_{\bP(\Om_{X_{t_0}})}(1)\otimes\pi_{\Om_{X_{t_0}}}^{*}\cO_{X_{t_0}}(-a))$ lies in the zero locus of this new section $\sigma_{t_0}^m\otimes D\otimes mE^{-1}$ whose divisorial component is $D$. Thus, since $|mE|$ is base point free, we know that the image must lie in the codimension at least two part of $Y_{t_0}$. This concludes the proof.



\subsection{Vector Fields}\label{globallygenerated}

As announced during the proof of the main result we are now going to prove the global generation property we used.

\begin{theorem}\label{global}
The bundle
$$T\bP(\Om_{\cX / \gP})\otimes \pi^*\rho_{1}^*\cO_{\bP^N}(N)\otimes \pi^*\rho_{2}^*\cO_{\gP}(1,\cdots,1)$$
is globally generated on $\bP(\Om_{\cX / \gP})$.
\end{theorem} 


The proof of this statement is, almost, the same as the proof of the main Theorem of \cite{Mer09}, so there is nothing here that wasn't already in Merker's paper. However since our situation is slightly different we still show how one can adapt Merker's computations here, and in particular we will point out the small differences, and where we are able to gain the better bound. This improvement in the bound is due to the fact that in the situation of \cite{Mer09}  the constructed vector fields have to satisfy many equations (as many as the dimension plus one) to be tangent to the higher order jet space, whereas in our situation we only need to go up to jets of order one, thus the constructed vector fields just have to satisfy two equations.  Also, for the reader's convenience, we adopt the notations of \cite{Mer09}.   


\subsubsection{Notations and Coordinates}

Now we fix homogeneous coordinates on $\bP^{N}$ and $\bP^{N_{d_i}}$ for any $1\leq i\leq c$.

\begin{eqnarray}
\left[Z\right]&=&\left[Z_0:\cdots : Z_N\right]\in \bP^N \nonumber \\
\left[A^{i}\right]&=&\left[(A^{i}_{\alpha})_{\alpha \in \bN^{N+1}, |\alpha |=d_i}\right]\in \bP^{N_{d_i} }\nonumber
\end{eqnarray}

In those coordinates $\cX=(F_{1}=0)\cap \cdots \cap (F_{c}=0)$ where

$$F_{i}=\sum_{\substack{\alpha \in \bN^{N+1} \\   |\alpha |=d_i}}A^{i}_{\alpha}Z^{\alpha} .$$
 
To construct vector fields explicitly it will be convenient to work in inhomogeneous coordinates. So from now on we suppose $Z_0 \neq 0$ and $A^{i}_{(0,d_i,0,\cdots,0)}\neq 0$ and we introduce the corresponding coordinates on $\bC^N$ and on $\bC^{N_{d_i}}$, by setting $z_i:=\frac{Z_i}{Z_0}$ and $a^{i}_{(\alpha_1,\cdots, \alpha_N)}=\frac{A^{i}_{(\alpha_0,\cdots ,\alpha_N)}}{A^{i}_{(0,d_i,0,\cdots,0)}}$ where $\alpha_0=d_i-\alpha_1-\cdots - \alpha_N$. Now in those coordinates the restriction $\cX_0$ of $\cX$ to the open subset $\bC^N\times\gP^0\subset \bP^N\times \gP$ where $\gP^0=\bC^{N_{d_1}}\times \cdots \times \bC^{N_{d_c}}$is defined by 

$$\cX_0=(f_{1}=0)\cap \cdots \cap (f_{c}=0)$$

where 
$$f_{i}=\sum_{\substack{\alpha \in \bN^{N} \\   |\alpha |\leq d_i}}a^{i}_{\alpha}z^{\alpha}.$$

On  $\bC^N\times \gP^0 \times \bC^N$ we will use the coordinates $(z_i,a^{1}_{\alpha},\cdots , a^{c}_{\alpha},z'_{k})$.
Now the equations defining the relative tangent bundle $T_{\cX/\gP^0}\subset \bC^N\times \gP^0 \times \bC^N$ are 

$$f_{i}=\sum_{\substack{\alpha \in \bN^{N} \\   |\alpha |\leq d_i}}a^{i}_{\alpha}z^{\alpha}=0$$
and
$$f'_{i}=\sum_{\substack{\alpha \in \bN^{N} \\   |\alpha |\leq d_i}}\sum_{k=1}^{N}a^{i}_{\alpha}\frac{\partial z^{\alpha}}{\partial z_k}z'_k=0.$$


\subsubsection{Vector Fields on $T_{\cX^0/\gP^0}$}

Let $\Sigma = \{ (z_i,a^{1}_{\alpha},\cdots , a^{c}_{\alpha},z'_{k})\ \ /\ \ (z'_1, \cdots , z'_N)\neq 0\}$.
Following \cite{Mer09} we are going to construct explicit vector fields on  $T_{\cX^0/\gP^0}$ (outside $\Sigma$) with prescribed pole order when we look at them as meromorphic vector 
fields on $T_{\cX/\bP}$. It will also be clear that the constructed vector fields can actually be viewed as vector fields on $\bP(\Om_{\cX/\gP})$.\\
A global vector field on $\bC^N\times \gP^0 \times \bC^N$ is of the form 

$$T=\sum_{j=1}^{N}Z_j\frac{\partial}{\partial z_j}+\sum_{\substack{\alpha \in \bN^{N} \\  |\alpha |\leq d_1}}A^{1}_{\alpha}\frac{\partial}{\partial a^{1}_{\alpha}}+\cdots +\sum_{\substack{\alpha \in \bN^{N} \\   |\alpha |\leq d_c}}A^{c}_{\alpha}\frac{\partial}{\partial a^{c}_{\alpha}}+\sum_{k=1}^{N}Z'_{k}\frac{\partial}{\partial z'_k}.$$

And such a vector field $T$ is tangent to $T_{\cX^0 /\gP^0}$ if for all $1\leq i \leq c$ 

\[
\left\{ {\begin{array}{ccc}
 T(f_{i}) & = & 0  \\
 T(f'_{i}) &  = &0  
 \end{array} } \right.
\]

First we construct for each $0\leq i \leq c$ vector fields of the form  
$$\sum_{\substack{\alpha \in \bN^{N} \\   |\alpha |\leq N}}A^{i}_{\alpha}\frac{\partial}{\partial a^{i}_{\alpha}}.$$

Such a vector field is tangent to $T_{\cX^0/\gP^0}$ if

$$T(f_{i})=\sum_{\substack{\alpha \in \bN^{N} \\   |\alpha |\leq N}}A^{i}_{\alpha}z^{\alpha} =  0  $$
and
$$T(f'_{i})=\sum_{\substack{\alpha \in \bN^{N} \\   |\alpha |\leq N}}\sum_{k=1}^{N}A^{i}_{\alpha}\frac{\partial z^{\alpha}}{\partial z_k}z'_k  = 0  .$$

As we are working outside $\Sigma$ we may as well suppose $z'_1\neq 0$. Set 
$$\cR_0(z,A)=\sum_{\substack{\alpha \in \bN^{N} \\   |\alpha |\leq N \\ \alpha \neq (0,\cdots ,0) \\ \alpha \neq (1,0,\cdots,0)}}A^{i}_{\alpha}z^{\alpha} \ \ and \ \ \cR_1(z,A)=\sum_{\substack{\alpha \in \bN^{N} \\   |\alpha |\leq N \\ \alpha \neq (0,\cdots ,0) \\ \alpha \neq (1,0,\cdots,0)}}\sum_{k=1}^{N}A^{i}_{\alpha}\frac{\partial z^{\alpha}}{\partial z_k}z'_k.$$

With those notation the tangency property is equivalent to solve the system

\[
\left\{ {\begin{array}{ccccccc}
 A^{i}_{(0,\cdots,0)} &+& A^{i}_{(1,0,\cdots,0)}z_1 &+& \cR_0(z,A) & = & 0  \\
                        & & A^{i}_{(1,0,\cdots,0)}z'_1 &+& \cR_1(z,A) & = & 0  
 \end{array} } \right.
\]
and as $z'_1\neq 0$ one can solve this in the straightforward way,

\[
\left\{ {\begin{array}{ccccc}
 A^{i}_{(1,0,\cdots,0)} &=&\frac{-1}{z'_1} \cR_1(z,A)   & &  \\
 A^{i}_{(0,\cdots,0)} &=& \frac{-z_1}{z'_1} \cR_1(z,A) &-& \cR_0(z,A)  \\                        
 \end{array} } \right.
\]

We see that the pole order of vector fields obtained this way is less than $N$ in the $z_i$'s.
This is where we get the improvement on the pole order. 
Now in order to span all the other directions, we can take the vector fields constructed by J. Merker, and the pole order of those fields will be less than $N$. 
For the reader's convinience we recall them here, without the proof, and refere to \cite{Mer09} for the details.\\  
First we recall how to construct vector fields of higher length in the $\frac{\partial}{\partial a^{i}_{\alpha}}$'s. For any $\alpha\in \bN^N$ such that $|\alpha |\leq d_i$ and any $\ell \in \bN^N$ such that $|\ell |\leq N$, set 

$$T_{\alpha}^{\ell}=\sum_{\substack{\ell' +\ell''=\ell \\ \ell',\ell''\in \bN^N}}\frac{\ell !}{\ell'!\ell''!}z^{\ell''}\frac{\partial}{\partial a_{\alpha-\ell}}.$$ 
Those vector fields are of order $N$ in is $z_i$'s, they are also tangent to $T_{\cX^0/\gP^0}$ and with the vector fields constructed before, they will span all the $\frac{\partial}{\partial a^{i}_{\alpha}}$ directions.\\

To span the $\frac{\partial}{\partial z_j}$ directions, for all $1\leq j\leq N$ we set 

$$T_j=\frac{\partial}{\partial z_j}-\sum_{|\alpha|\leq d_1-1}a^{1}_{\alpha +e_j}(\alpha_j +1)\frac{\partial}{\partial a^{1}_{\alpha}} \cdots -\sum_{|\alpha|\leq d_c-1}a^{c}_{\alpha +e_j}(\alpha_j +1)\frac{\partial}{\partial a^{c}_{\alpha}}$$

where $e_j=(0,\cdots,0,1,0,\cdots, 0)$ is the $N-$uple where the only non zero term is in slot $j$. It is now a straightforward to check that those vector fields are tangent to $T_{\cX^0/\gP^0}$. And they are of order $1$ in each of the $a^{i}_{\alpha}$'s.\\
We recall now how to span the $\frac{\partial}{\partial z'_k}$ direction. For any $(\Lambda^{\ell}_k)^{1\leq \ell \leq N}_{1\leq k \leq N}\in GL_N(\bC)$   we look for vect fields of the form 

$$T_{\Lambda}=\sum_{k=1}^{N}\left( \sum_{\ell=1}^{N} \Lambda^{\ell}_{k}z'_{\ell} \right)\frac{\partial}{\partial z'_k} +\sum_{|\alpha| \leq d_1}A^{1}_{\alpha}(z,a,\Lambda)\frac{\partial}{\partial a^{1}_{\alpha}}+ \cdots + \sum_{|\alpha| \leq d_c}A^{c}_{\alpha}(z,a,\Lambda)\frac{\partial}{\partial a^{c}_{\alpha}}.$$

Now for each such vector field and for any $1\leq i\leq c$ set 

$$T^{i}_{\Lambda}=\sum_{k=1}^{N}\left( \sum_{\ell=1}^{N} \Lambda^{\ell}_{k}z'_{\ell} \right)\frac{\partial}{\partial z'_k} +\sum_{|\alpha| \leq d_i}A^{i}_{\alpha}(z,a,\Lambda)\frac{\partial}{\partial a^{i}_{\alpha}}.$$

Now we can easily check that we have

\[
\left\{ {\begin{array}{ccc}
 T_{\Lambda}(f_{i}) &=&T^{i}_{\Lambda}(f_{i}) \\
 T_{\Lambda}(f'_{i}) &=&T^{i}_{\Lambda}(f'_{i})  \\                        
 \end{array} } \right.
\]

Therefore to construct vector fields tangent to $T_{\cX^0/\gP^0}$ we can solve those equations for each $1\leq i\leq c$ independently, but this can be done using Merker's result on this type of vector fields. By doing so we find for each $1\leq i \leq c$ solutions of the type

$$ A^{i}_{\alpha}(z,a,\Lambda)=\sum_{|\beta|<N}\mathcal{L}^{\beta}_{\alpha,i}(a,\Lambda)z^{\beta} $$

where  $\mathcal{L}^{\beta}_{\alpha,i}$ is bilinear in $(a,\Lambda)$. Therefore the the constructed fields will have order less than $N$ in the $z_j$'s and of order $1$ in each of the $a^{i}_{\alpha}$. We refer again to \cite{Mer09} for proof of those facts. This leads to the desired result.


\subsection{Effective Existence of Symmetric Differentials}\label{effectifsymmetric}

The special case when $\kappa=1$ in Theorem \ref{thmjet} (that is, when the codimension is greater than the dimension) is of particular interest to us here, therefore  we will give an effective bound on the degree in this case. This will allow us to give an explicit bound on the degree in our main results. First we rewrite this theorem in the present situation.


\begin{theorem}\label{nonvanishingeff}
Fix $a\in\mathbb{N}$. Then there exists a constant $\Gamma_{N,n,a}$ such that if
$X\subset \mathbb{P}^N$ is a smooth complete intersection of dimension $n$, codimension $c$ and multidegree $(d_1,\cdots,d_c)$ satisfying $n\leq c$ and $d_i\geq \Gamma_{N,n,a}$ then $\mathcal{O}_{\mathbb{P}(\Omega_X)}(1)\otimes\mathcal{O}_{X}(-a)$ is big. In particular, when $m\gg 0$,
$$H^0(X,S^m\Omega_X\otimes \mathcal{O}_X(-am))\neq 0.$$ 
\end{theorem}

We give a rough bound on $\Gamma_{N,n,a}$ that works for any $N,n,a$ and afterwards we  give a better bound when $n=2$.

\begin{remark}
We would like to mention that O. Debarre proved in \cite{Deb05}, using Riemann-Roch computations, that $\cO_{\bP(\Om_X)}(1)\otimes \cO_X(1)$ is big under the assumptions on the dimension and the degree.
\end{remark}


\subsubsection{Segre Classes}

We start by giving the detailed expression of the segre classes we estimated previously.
Recall that we had the formula
\begin{eqnarray}
s(\Omega_X)=(1- h  + h^2-...)^{N+1}\prod_{i=1}^c(1\!+\! d_i h). \nonumber
\end{eqnarray}
Let us introduce another notation, let 
$$\epsilon_i(d_1,\cdots ,d_c) := \sum_{1\leq j_1<...<j_i\leq c} d_{j_1}...d_{j_i}.$$ 
And also, $\epsilon_i(d_1,\cdots ,d_c)=0$ if $i>c$.
Now, by expending and since $h^k=0$ if $k>n$  we get,

\begin{eqnarray}
s(\Omega_X)&=&\left( \sum_{k=0}^{n}\binom{N+k}{k}(-1)^kh^k \right)\left(\sum_{i=0}^c \epsilon_i(d_1,\cdots ,d_c)h^i\right)\nonumber \\
&=& \sum_{k=0}^{n} \sum_{i=0}^{n}\binom{N+k}{k} (-1)^k \epsilon_{i}(d_1,\cdots,d_c)h^{i+k} \nonumber \\
&=& \sum_{j=0}^{n} \sum_{k=0}^{j}\binom{N+k}{k}  (-1)^k \epsilon_{j-k}(d_1,\cdots,d_c)h^{j} \nonumber
\end{eqnarray}

And from this we deduce the explicit form of the Segre classes that will be used afterwards,

\begin{eqnarray}\label{segre}
s_j(\Omega_X)=\sum_{k=0}^{j}\binom{N+k}{N}  (-1)^k \epsilon_{j-k}(d_1,\cdots,d_c)h^{j}
\end{eqnarray}

\subsubsection{Explicite Intersection Computations}

Here we will do explicitly, in this context, the intersection computations we did during the proof of \ref{thmjet}. We take the same notations, under the assumption $n\leq c$.

\begin{eqnarray}
 F^{2n-1}&-&(2n-1)F^{2n-2}\cdot G \nonumber \\
 &=& \int_{\bP(\Om_X)}(u+2h)^{2n-1}-(2n-1)(u+2h)^{2n-2}(2+a)h \nonumber \\
 &=& \int_{\bP(\Om_X)}\sum_{i=0}^{2n-1}\binom{2n-1}{i}u^{2n-1-i}(2h)^i - (2n-1)(2+a)\sum_{j=0}^{2n-2} \binom{2n-2}{j}u^{2n-2-j}(2h)^j h \nonumber \\
 &=& \int_{\bP(\Om_X)}u^{2n-1}+\sum_{i=1}^{2n-1}\left(2^i \binom{2n-1}{i} - (2n-1)(2+a)2^{i-1} \binom{2n-2}{i-1}\right) u^{2n-1-i}h^i \nonumber \\  
 &=& \int_{\bP(\Om_X)}u^{2n-1}+\sum_{i=1}^{2n-1} 2^{i-1}(2-i(2+a))\binom{2n-1}{i}u^{2n-1-i}h^i\nonumber \\
 &=& \int_{X} \sum_{i=0}^{n} 2^{i-1}(2-i(2+a))\binom{2n-1}{i}s_{n-i}h^i\nonumber
\end{eqnarray}
 
Now we will use formula (\ref{segre}) to see how this intersection product depend on the multidegree $(d_1,\cdots ,d_c)$. To ease the notations we will also just write $\epsilon_i$ instead of $\epsilon_i(d_1,\cdots,d_c)$. 

\begin{eqnarray}
 F^{2n-1}&-&(2n-1)F^{2n-2}\cdot G \nonumber \\
 &=& \sum_{i=0}^{n} \int_{X} 2^{i-1}(2-i(2+a))\binom{2n-1}{i}s_{n-i}h^i\nonumber \\
 &=& \sum_{i=0}^{n}  2^{i-1}(2-i(2+a))\binom{2n-1}{i}\left( \sum_{k=0}^{n-i}\binom{N+k}{N}  (-1)^k \epsilon_{n-i-k}\right) \nonumber \\
 &=& \sum_{i=0}^{n}  \sum_{k=0}^{n-i} (-1)^k  2^{i-1}(2-i(2+a))\binom{2n-1}{i}\binom{N+k}{N}\epsilon_{n-i-k} \nonumber \\
 &=&  \sum_{j=0}^{n} \sum_{i=0}^{n-j}  (-1)^{n-i-j}  2^{i-1}(2-i(2+a))\binom{2n-1}{i}\binom{N+n-i-j}{N}\epsilon_{j} \nonumber \\
 &=&  \sum_{j=0}^{n} \D_{a}^{N,n,j} \epsilon_{j}(d_1,\cdots ,d_c).\nonumber
\end{eqnarray}
  
Where $$\D_{a}^{N,n,j}= (-1)^{n-j}\sum_{i=0}^{n-j}  (-1)^{i}  2^{i-1}(2-i(2+a))\binom{2n-1}{i}\binom{N+n-i-j}{N} $$  
  
\subsubsection{Rough Bound}

Now we will be able to give a straightforward rough bound for $\Gamma_{N,n,a}$ for any $N,n,a$. 
Let us recall a basic fact to estimate the zero locus of a polynomial in one real variable.

\begin{lemma}\label{estimate}
Let $P(x):=x^k+a_{k-1}x^{k-1}+\cdots +a_1x+a_0\in \mathbb{R}[x]$. If $x\geq 1+\max_{i} |a_i| $ then $P(x)>0$
\end{lemma}  
Now we will see how to apply this in our situation. As previously we write

$$\epsilon_i(x_1,\cdots ,x_c) := \sum_{1\leq j_1<...<j_i\leq c} x_{j_1}...x_{j_i}.$$ 

Now take $k\leq c$ and
$$P(x_1,\cdots ,x_c)=\epsilon_k(x_1,\cdots ,x_c)+a_{k-1}\epsilon_{k-1}(x_1,\cdots ,x_c)+\cdots +a_1\epsilon_1(x_1,\cdots ,x_c)+a_0.$$

We want to find a $r\in \mathbb{R}$ such that if $x_i\geq r$ for all $1\leq i \leq c$ then $P(x_1,\cdots,x_c)>0$. We will be done if we are able to find $r\in \mathbb{R}$ satisfying $P(r,\cdots,r)>0$ and $\frac{\partial P}{\partial x_i}(x_1,\cdots,x_c)>0$ for all $1\leq i \leq c$ as soon as $x_j\geq r$ for all $1\leq i \leq c$. Now observe that since the $\epsilon_i$'s are symmetric, we have to check the positivity of just one partial derivative, let's say $\frac{\partial P}{\partial x_c}$.
By induction, we are left to find $r\in \mathbb{R}$ satisfying $P(r,\cdots,r)>0$, $\frac{\partial P}{\partial x_c}(r,\cdots ,r)>0$, $\frac{\partial }{\partial x_{c-1}}\frac{\partial P}{\partial x_c}(r,\cdots ,r)>0$, $\cdots$,  $\frac{\partial }{\partial x_1}\cdots \frac{\partial }{\partial x_{c-1}}\frac{\partial P}{\partial x_c}(r)>0$. But we also have 
$$\frac{\partial \epsilon_j}{\partial x_c}(x_1,\cdots,x_c)=\epsilon_{j-1}(x_1,\cdots,x_{c-1}). $$
and thus 

$$\frac{\partial P}{\partial x_c} =\frac{\partial }{\partial x_c} \sum_{i=0}^k a_{i}\epsilon_{i}(x_1,\cdots ,x_c)=\sum_{i=0}^{k-1} a_{i+1} \epsilon_{i}(x_1,\cdots,x_{c-1}). $$
Where $a_k=1$. And similarly 

$$\frac{\partial }{\partial x_(c-j)}\cdots \frac{\partial }{\partial x_{c-1}}\frac{\partial P}{\partial x_c}=\sum_{i=0}^{k-j} a_{i+j} \epsilon_{i}(x_1,\cdots,x_{c-j}).$$

Evaluating in $(r,\cdots,r)$ yields for any $0\leq j\leq c$
$$\frac{\partial }{\partial x_{(c-j)}}\cdots \frac{\partial }{\partial x_{c-1}}\frac{\partial P}{\partial x_c}(r,\cdots ,r)=\sum_{i=0}^{k-j} a_{i+j} \binom{c-j}{i} r^i.$$

Now we apply lemma \ref{estimate} to each of those polynomials and therefore we just have to give a bound for 

$$\max_{\substack{0\leq j \leq k-1 \\ j \leq i \leq k-1}} \left| a_i \frac{\binom{c-j}{i-j}}{\binom{c-j}{k-j}}\right|.$$

but since
$$\frac{\binom{c-j}{i-j}}{\binom{c-j}{k-j}} \leq \frac{\binom{c}{i}}{\binom{c}{k}} $$

we just have to give a bound for

$$\max_{ 0 \leq i \leq k-1} \left| a_i \frac{\binom{c}{i}}{\binom{c}{k}}\right|.$$

This is what we will do now in our intersection product computation, that is when $k=n$, $c=N-n$ and $a_j=\D_{a}^{N,n,j}$.
\begin{eqnarray}
\left| \D_{a}^{N,n,j}\frac{\binom{c}{j}}{\binom{c}{n}}\right| &=& \left|\sum_{i=0}^{n-j}  (-1)^{n-i-j}  2^{i-1}(2-i(2+a))\binom{2n-1}{i}\binom{N+n-i-j}{N}\frac{\binom{c}{j}}{\binom{c}{n}}\right|\nonumber \\
&\leq & \binom{N+n-j}{N}\frac{\binom{c}{j}}{\binom{c}{n}}+ \sum_{\substack{i=1}}^{n-j} 2^{i-1}(i(2+a)-2)\binom{2n-1}{i}\binom{N+n-i-j}{N}\frac{\binom{c}{j}}{\binom{c}{n}}\nonumber \\
&\leq & \binom{N+n-j}{N}\frac{\binom{c}{j}}{\binom{c}{n}}+  2^{n-j-1}((n-j)(2+a)-2)\sum_{\substack{i=1 }}^{n-j}\binom{2n-1}{i}\binom{N+n-i-j}{N}\frac{\binom{c}{j}}{\binom{c}{n}}\nonumber \\
&\leq & \binom{N+n-j}{N}\frac{\binom{c}{j}}{\binom{c}{n}}\nonumber \\
&+&  2^{n-j-1}((n-j)(2+a)-2)(n-j)\binom{N+n-j-1}{N}\frac{\binom{c}{j}}{\binom{c}{n}}\binom{2n-1}{n-j}\nonumber \\
&\leq & \binom{n}{j}\frac{(N+n)!(N-2n)!}{N!(N-n)!}\nonumber \\
&+&  2^{n-j-1}((n-j)(2+a)-2)(n-j)\binom{n}{j}\binom{2n-1}{n-j}\frac{(N+n)!(N-2n)!(n-j)}{N!(N-n)!(N+n-j)}\nonumber \\
&\leq &  \left( 2^{n-1}(n(2+a)-2)\frac{n^2}{(N+1)}\binom{2n-1}{n}+1\right)\binom{n}{\lfloor\frac{n}{2}\rfloor}\frac{(N+n)!(N-2n)!}{N!(N-n)!}. \nonumber 
\end{eqnarray}

This gives the desired (rough) bound for $\Gamma_{N,n,a}$,

\begin{eqnarray}
\Gamma_{N,n,a}\leq \left( 2^{n-1}(n(2+a)-2)\frac{n^2}{(N+1)}\binom{2n-1}{n}+1\right)\binom{n}{\lfloor\frac{n}{2}\rfloor}\frac{(N+n)!(N-2n)!}{N!(N-n)!}.\nonumber
\end{eqnarray}

Obviously this is far from optimal, but we won't go into more details for the general case. However even with such an estimate we can make a noteworthy remark. In our main theorem we use $\Gamma_{N,n,a+N}$. Now if we fix $n$ and if we let go $N$ to infinity (that is when the codimension becomes bigger and bigger) we get

$$\lim_{N\to+\infty}\Gamma_{N,n,N+a}\leq 2^{n-1}n^3\binom{2n-1}{n}\binom{n}{\lfloor\frac{n}{2}\rfloor}.$$   

That is to say that this bound decreases as $c$ gets bigger, and will have a limit depending only on $n$. This corresponds to the intuition that as the codimension increases the situation get's improved and the multidegree can be taken smaller. This feature should be a part of any bound found on $\Gamma$.  


\subsubsection{Bound in Dimension Two}\label{dim2}


Here we give a better bound in the case of surfaces, which is of particular interest to us since it is the case where we will have the strongest conclusion. Take the notation of the previous section, and let $n=2$ so that $c=N-2$. Fix $a\in \bN$ (the case $a=N$is the most important). We want to estimate
$F^{3}-3F^{2}\cdot G=\sum_{j=0}^{2} \D_{a}^{N,2,j} \epsilon_{j}(d_1,\cdots ,d_c)$. Where,
\begin{eqnarray}
\D_{a}^{N,2,2}&=& 1\nonumber \\
\D_{a}^{N,2,1}&=& -(N+1)-3a\nonumber \\
\D_{a}^{N,2,0}&=& \binom{N+2}{N}+3a(N+1)-12(a+1).\nonumber
\end{eqnarray}

Observe that $\D_{a}^{N,2,0}\geq 0$ when $N\geq 4$, thus $F^{3}-3F^{2}\cdot G\leq \epsilon_2(d_1,\cdots ,d_c)-\D_{a}^{N,2,1}\epsilon_1(d_1,\cdots ,d_c)$, and therefore we just have to bound one term,
$$\left| \D_{a}^{N,2,1} \frac{\binom{N-2}{1}}{\binom{N-2}{2}} \right|=2\frac{N+1+3a}{N-3}.$$

Thus $\Gamma_{N,2,a}\leq 2\frac{N+1+3a}{N-3}$. And in particular $\Gamma_{N,2,N}\leq \frac{8N+2}{N-3}$.\\ 
\\
\textit{Acknowledgement}. The author wishes to thank C. Mourougane for his many advices, and S. Diverio for very helpful discussions and comments.





\end{document}